\documentclass[a4paper,12pt,reqno,oneside]{amsart}

\usepackage[utf8x]{inputenc}
\usepackage[T1]{fontenc}
\usepackage{geometry}
\usepackage{lmodern}

\usepackage{amsmath}
\usepackage{amssymb}
\usepackage{amsthm}
\usepackage[matrix,arrow,cmtip]{xy}
\usepackage{enumerate}
\usepackage[bookmarksnumbered,colorlinks,linkcolor=black,citecolor=black,urlcolor=black]{hyperref}
\usepackage[capitalise]{cleveref}
\usepackage[foot]{amsaddr}

\newcommand{\A}{\mathbb{A}}
\newcommand{\C}{\mathbb{C}}
\newcommand{\G}{\mathbb{G}}
\newcommand{\HH}{\mathbb{H}}
\newcommand{\bP}{\mathbb{P}}
\newcommand{\Q}{\mathbb{Q}}
\newcommand{\R}{\mathbb{R}}
\newcommand{\Z}{\mathbb{Z}}

\newcommand{\Ag}{\mathcal{A}_g}
\newcommand{\Fg}{\mathcal{F}_g}
\newcommand{\Hg}{\mathcal{H}_g}

\newcommand{\cA}{\mathcal{A}}
\newcommand{\cB}{\mathcal{B}}
\newcommand{\cN}{\mathcal{N}}
\newcommand{\cW}{\mathcal{W}}

\newcommand{\fo}{\mathfrak{o}}

\newcommand{\rM}{\mathrm{M}}

\DeclareMathOperator{\End}{End}
\DeclareMathOperator{\Gal}{Gal}
\DeclareMathOperator{\GL}{GL}
\DeclareMathOperator{\GSp}{GSp}
\DeclareMathOperator{\Hom}{Hom}
\DeclareMathOperator{\Nm}{N}
\DeclareMathOperator{\Nrd}{Nrd}
\DeclareMathOperator{\Res}{Res}
\DeclareMathOperator{\Sh}{Sh}
\DeclareMathOperator{\Sp}{Sp}
\DeclareMathOperator{\Spec}{Spec}

\renewcommand{\Im}{\mathop{\mathrm{Im}}}

\newcommand{\abs}[1]{\left\lvert #1 \right\rvert}
\newcommand{\ad}{\mathrm{ad}}
\newcommand{\ca}{\mathrm{ca}}

\newcommand{\fullsmallmatrix}[4]{ \bigl( \begin{smallmatrix} #1 & #2 \\ #3 & #4 \end{smallmatrix} \bigr) }

\newcommand{\defterm}{\textbf}

\newtheorem{lemma}{Lemma}[section]
\newtheorem{proposition}[lemma]{Proposition}
\newtheorem{theorem}[lemma]{Theorem}
\newtheorem{conjecture}[lemma]{Conjecture}
\Crefname{conjecture}{Conjecture}{Conjectures} 

\newcommand{\tp}[1]{{}^t {#1}}
\newcommand{\Ztil}{\widetilde{Z}}
\newcommand{\Ltil}{\widetilde{\Lambda}}
\newcommand{\stil}{\tilde{s}}
\newcommand{\ttil}{\tilde{t}}
\newcommand{\util}{\tilde{u}}

\author{Martin Orr}
\address{Université Paris-Sud 11\\ Bat.\ 425\\ 91400 Orsay\\ France}
\email{martin.orr@math.u-psud.fr}

\title[Families of abelian varieties]{Families of abelian varieties with many isogenous fibres}

\begin{document}

\begin{abstract}
Let \( Z \) be a subvariety of the moduli space of principally polarised abelian varieties of dimension~\( g \) over the complex numbers.
Suppose that \( Z \) contains a Zariski dense set of points which correspond to abelian varieties from a single isogeny class.
A generalisation of a conjecture of André and Pink predicts that \( Z \) is a weakly special subvariety.
We prove this when \( \dim Z = 1 \) using the Pila--Zannier method and the Masser--Wüstholz isogeny theorem.
This generalises results of Edixhoven and Yafaev when the Hecke orbit consists of CM points and of Pink when it consists of Galois generic points.
\end{abstract}

\maketitle

\section{Introduction}

Let \( \Ag \) denote the Siegel moduli space of principally polarised abelian varieties of dimension~\( g \).
We consider the following conjecture.
In particular we prove the conjecture when \( Z \) is a curve, and make some progress on higher--dimensional cases.

\begin{conjecture} \label{main-conj}
Let \( \Lambda \) be the isogeny class of a point \( s \in \Ag(\C) \).
Let \( Z \) be an irreducible closed subvariety of \( \Ag \) such that \( Z \cap \Lambda \) is Zariski dense in \( Z \).
Then \( Z \) is a weakly special subvariety of \( \Ag \).
\end{conjecture}

\begin{theorem} \label{main-thm-curves}
\Cref{main-conj} holds when \( Z \) is a curve.
\end{theorem}

\begin{theorem} \label{main-thm-higher}
Let \( \Lambda \) be the isogeny class of a point \( s \in \Ag(\C) \).
Let \( Z \) be an irreducible closed subvariety of \( \Ag \) such that \( Z \cap \Lambda \) is Zariski dense in \( Z \).

Then there is a special subvariety \( S \subset \Ag \) which is isomorphic to a product of Shimura varieties \( S_1 \times S_2 \) with \( \dim S_1 > 0 \), and such that 
\[ Z = S_1 \times Z' \subset S \]
for some irreducible closed subvariety \( Z' \subset S_2 \).
\end{theorem}

\Cref{main-thm-higher}, but not \cref{main-thm-curves}, depends on results concerning the hyperbolic Ax--Lindemann conjecture from recent preprints of Pila and Tsimerman~\cite{pila-tsimerman:ax-lindemann} and of Ullmo~\cite{ullmo:hyp-ax-lind}.

\Cref{main-conj} is a consequence of the Zilber--Pink conjecture on subvarieties of Shimura varieties~\cite{pink:zp}.
For a statement of the Zilber--Pink conjecture and proof that it implies \cref{main-conj}, see section~\ref{sec:shimura-vars}.

\Cref{main-conj} is slightly more general than the \( S = \Ag \) case of the following conjecture of André and Pink, because the isogeny class of \( s \in \Ag(\C) \) is sometimes bigger than the Hecke orbit:
by \defterm{isogeny class} we mean the set of points \( t \in \Ag(\C) \) such that the corresponding abelian variety \( A_t \) is isogenous to \( A_s \), with no condition of compatibility between isogeny and polarisations.
On the other hand the \defterm{Hecke orbit} consists of those points for which there is a \defterm{polarised isogeny} between the principally polarised abelian varieties -- that is, an isogeny \( \phi \colon A_s \to A_t \) satisfying \( \phi^* \lambda_t \in \Z\lambda_s \), where \( \lambda_s \) and \( \lambda_t \) are the polarisations.
In the case of \( \Ag \) there is no difference between Hecke orbits and Pink's generalised Hecke orbits.

\begin{conjecture} \label{pink-conj}
[\cite{andre:g-functions}~Chapter~X Problem~3, \cite{pink:conj}~Conjecture~1.6]
Let \( S \) be a mixed Shimura variety over \( \C \) and \( \Lambda \subset S \) the generalised Hecke orbit of a point \( s \in S \).
Let \( Z \subset S \) be an irreducible closed algebraic subvariety such that \( Z \cap \Lambda \) is Zariski dense in \( Z \).
Then \( Z \) is a weakly special subvariety of \( S \).
\end{conjecture}

Some cases of \cref{main-conj} are already known:
If the point \( s \) is Galois generic, then the isogeny class and the Hecke orbit coincide, and \cref{main-conj} follows from equidistribution results of Clozel, Oh and Ullmo, as was shown by Pink~\cite{pink:conj}.
When \( s \) is a special point, \cref{main-thm-curves} was proved by Edixhoven and Yafaev~\cite{edixhoven-yafaev} by exploiting the fact that Galois orbits of special points are contained in \( Z \cap T_g Z \) for suitable Hecke operators \( T_g \) and these Galois orbits tend to be large compared to the degree of \( T_g \).
When \( s \) corresponds to a product of elliptic curves, Habegger and Pila~\cite{habegger-pila:unlikely} proved the theorem using the method we extend here.

The terminology ``weakly special subvariety'' was introduced by Pink~\cite{pink:conj},
although the concept was first studied by Moonen~\cite{moonen:linearity-I}.
Moonen showed that a subvariety of a Shimura variety is totally geodesic (in the sense of differential geometry) if and only if it satisfies the following definition.
An algebraic subvariety \( Z \) of \( \Ag \) is called \defterm{weakly special} if there exist a sub-Shimura datum \( (H, X_H) \) of \( (\GSp_{2g}, \Hg^\pm) \), a decomposition
\[ (H^\ad, X_H^\ad) = (H_1, X_1) \times (H_2, X_2) \]
and a point \( x_2 \in X_2 \) such that \( Z \) is the image in \( \Ag \) of \( X_1 \times \{ x_2 \} \).
In other words, to say that \( Z \) is weakly special means that we can choose \( S \), \( S_1 \), \( S_2 \) in the conclusion of \cref{main-thm-higher} such that \( Z' \) is a single point in~\( S_2 \).
For more details, see section~\ref{subsec:special-subvars}.

In this article we will use a characterisation of weakly special subvarieties due to Ullmo and Yafaev~\cite{uy:characterisation}: \( Z \) is weakly special if and only if an irreducible component of \( \pi^{-1}(Z) \) is algebraic, where \( \pi \) is the quotient map \( \Hg \to \Ag \) and \( \Hg \subset \rM_{2g \times 2g}(\C) \) is the Siegel upper half space.
Here we call a subvariety of \( \Hg \) \defterm{algebraic} if it is a connected component of \( W_0 \cap \Hg \) for some algebraic variety \( W_0 \subset \rM_{2g \times 2g}(\C) \).
In order to prove \Cref{main-thm-higher} we require a strengthening of this characterisation called the hyperbolic Ax--Lindemann conjecture:
if \( W \) is a maximal algebraic subvariety of \( \pi^{-1}(Z) \), then \( \pi(W) \) is algebraic.
A proof of the Ax--Lindemann conjecture for \( \Ag \) was recently announced by Pila and Tsimerman~\cite{pila-tsimerman:ax-lindemann}.

\medskip

Our proof of \cref{main-thm-curves,main-thm-higher} follows the method proposed by Pila and Zannier for proving the Manin--Mumford and André--Oort conjectures~\cite{pila-zannier}.
This is based upon counting rational points of bounded height in certain analytic subsets of \( \Hg \), and applying the Pila--Wilkie counting theorem on sets definable in o-minimal structures.

The central part of the proof of \cref{main-thm-curves,main-thm-higher} is in section~\ref{sec:main-thm}.
This uses a strong version of the Pila--Wilkie counting theorem involving definable blocks.
The other ingredients are an upper bound for the heights of matrices in \( \GL_{2g}(\Q) \) relating isogenous points, proved in section~\ref{sec:matrix-heights}, and a lower bound for the Galois degrees of principally polarised abelian varieties in an isogeny class, derived from the Masser--Wüstholz isogeny theorem~\cite{mw:isogeny-avs}.

In section~\ref{sec:isog-fg-fields} we use a specialisation argument to prove a version of the Masser--Wüstholz isogeny theorem for finitely generated fields of characteristic~\( 0 \), generalising the original theorem which was valid only over number fields.
This is necessary in order to prove \cref{main-thm-curves,main-thm-higher} for points \( s \) and subvarieties \( Z \) defined over \( \C \) and not only over \( \bar{\Q} \).

\medskip

Now we consider some generalisations of \cref{main-thm-curves}.
\Cref{main-thm-curves} immediately implies \cref{pink-conj} for curves \( Z \) in Shimura varieties \( S \) of Hodge type, if we restrict to usual Hecke orbits.
This is because, by the definition of a Shimura variety of Hodge type, there is a finite morphism \( f \colon S \to \Ag \)
for some \( g \) such that the image of each Hecke orbit in \( S \) is contained in a Hecke orbit of \( \Ag \), and \( Z \subset S \) is weakly special if and only if \( f(Z) \subset \Ag \) is weakly special.
However this does not imply \cref{pink-conj} for generalised Hecke orbits in Shimura varieties of Hodge type, as a generalised Hecke orbit in \( S \) may map into infinitely many isogeny classes in~\( \Ag \).

Because of the use of the Masser--Wüstholz theorem, our method applies only to Shimura varieties parameterising abelian varieties i.e.\ those of Hodge type.
In particular, let us compare with Theorem~1.2 of \cite{edixhoven-yafaev}.
Take any Shimura datum~\( (G, X) \).
Edixhoven and Yafaev generalise Hecke orbits by choosing a representation of \( G \) and considering a set of points where the induced \( \Q \)-Hodge structures are isomorphic.
The Masser--Wüstholz theorem can be used only when these Hodge structures have type \( (-1,0)+(0,-1) \).
Hence our method lacks a key advantage of Edixhoven and Yafaev's formulation, namely that they can replace \( G \) by a subgroup so that \( Z \) is Hodge generic, or by its adjoint group.

This restriction to isogeny classes of abelian varieties rather than generalised Hecke orbits is related to our inability to prove the full \cref{main-conj}.
In the case of the André--Oort conjecture, a conclusion as in \cref{main-thm-higher} implies the full conjecture by induction on \( \dim Z \)  (see~\cite{ullmo:hyp-ax-lind}).
This is because, when \( Z \) is of the form \( S_1 \times Z' \), special points in \( Z \) project to special points in \( Z' \).

This does not work for \cref{main-conj} because the hypothesis that \( Z = S_1 \times Z' \) contains a dense set of points from a single isogeny class does not imply the same thing for \( \{ x_1 \} \times Z' \), where we fix a point \( x_1 \in S_1 \) in order to realise \( Z' \) as a subvariety of \( \Ag \).
The problem is that the decomposition \( S = S_1 \times S_2 \) need not have an interpretation in terms of moduli of abelian varieties.
For example this happens in André and Borovoi's example of a subvariety \( S \subset \cA_8 \) which decomposes as a product of Shimura varieties but where the generic abelian variety in the family parameterised by \( S \) is simple.

\subsection*{Acknowledgements}
I am grateful to Emmanuel Ullmo for suggesting to me the problem treated in this paper and for regular conversations during its preparation.
I would also like to thank Barinder Banwait for his comments on an earlier version of the manuscript,
and Gaël Rémond for remarks on \cref{isogeny-bound}.
I am grateful to the referee for their detailed comments.

This paper was published in \textit{Journal für die reine und angewandte Mathematik (Crelles Journal)}, 2015, Issue 705, p.~211--231 (DOI: \href{http://dx.doi.org/10.1515/crelle-2013-0058}{\texttt{10.1515/crelle-2013-0058}}).
The published version is available at \href{http://www.degruyter.com/view/j/crelle.2015.2015.issue-705/crelle-2013-0058/crelle-2013-0058.xml}{\texttt{www.degruyter.com}}.

There is a gap in the proof of \cref{matrix-height-bound-fg} in the published version of this paper.
This was discovered by Gabriel Dill during his ongoing PhD studies (2016).
I am very grateful to Gabriel both for pointing out this gap and for finding a method of fixing it, which has been added to the arXiv version of the paper as \cref{correction:rat-rep-height-bound-plus}.

\section{Shimura varieties and the Zilber--Pink conjecture} \label{sec:shimura-vars}

In this article we consider only the moduli space of abelian varieties and its subvarieties, however to place \cref{main-conj} in its proper context we need to consider Shimura varieties.
For the convenience of the reader, we briefly summarise the theory of Shimura varieties and the Zilber--Pink conjecture in this section.
This contains no original material:
the primary sources are \cite{deligne:shimura-vars} for Shimura varieties, \cite{moonen:linearity-I} for special and weakly special subvarieties (special subvarieties are what Moonen calls subvarieties of Hodge type) and \cite{pink:zp} for the Zilber--Pink conjecture.
We also prove that the Zilber--Pink conjecture implies \cref{main-conj} by an argument which is essentially due to Pink.

\subsection{Shimura varieties}

A \defterm{Shimura datum} is defined to be a pair \( (G, X) \) where \( G \) is a reductive algebraic group over \( \Q \) and \( X \) is a \( G(\R) \)-conjugacy class of homomorphisms
\[ h \colon \Res_{\C/\R} \G_m \to G_\R \]
satisfying the following conditions:

\begin{enumerate}[(SV1)]
\item the Hodge structure on the adjoint representation of \( G \) induced by \( h \) has type contained in \( \{ (-1,1), (0,0), (1,-1) \} \);
\item \( \operatorname{int} h(i) \) induces a Cartan involution of the adjoint group of \( G(\R) \);
\item \( G^\ad \) has no factor defined over \( \Q \) on which the projection of \( h \) is trivial.
\end{enumerate}

Under these conditions, each connected component of \( X \) is a Hermitian symmetric domain on which the identity component \( G(\R)^+ \) of \( G(\R) \) acts holomorphically.

Let \( K \) be a compact open subgroup of \( G(\A_f) \), where \( \A_f \) is the ring of finite adeles.
We define
\[ \Sh_K(G, X) = G(\Q) \backslash X \times G(\A_f) / K, \]
where \( G(\Q) \) acts diagonally on \( X \times G(\A_f) \) on the left and \( K \) acts on \( G(\A_f) \) only on the right.
Deligne~\cite{deligne:shimura-vars} showed that \( \Sh_K(G, X) \) can be given the structure of an algebraic variety over a number field, called a \defterm{Shimura variety}.

Choose a connected component \( X^+ \) of \( X \).
Let \( G(\R)_+ \) be the preimage of the identity component (in the analytic topology) of \( G^\ad(\R) \) in \( G(\R) \); this is the stabiliser of \( X^+ \).
Let \( G(\Q)_+ = G(\R)_+ \cap G(\Q) \).

The image of \( X^+ \times \{ 1 \} \subset X \times G(\A_f) \) in \( \Sh_K(G, X)_\C \) is called the \defterm{neutral component} of \( \Sh_K(G, X)_\C \).
As a complex manifold it is canonically isomorphic to
\[ \Gamma \backslash X^+ \]
where \( \Gamma = K \cap G(\Q)_+ \) is a congruence subgroup of \( G(\Q)_+ \).

Let \( (G, X) \) be a Shimura datum and let \( G^\ad = G/Z(G) \) be the adjoint group of \( G \), where \( Z(G) \) is the centre of \( G \).
We get a new Shimura datum \( (G^\ad, X^\ad) \) by letting \( X^\ad \) be the \( G^\ad(\R) \)-conjugacy class of morphisms \( \Res_{\C/\R} \G_m \to G^\ad_\R \) containing \( \pi \circ x \) for \( x \in X \), where \( \pi \) is the quotient map \( G \to G^\ad \).
The map \( X \to X^\ad \) is an injection whose image is a union of connected components of \( X^\ad \) (\cite{moonen:linearity-I} section~2.1).

\subsection{Moduli of abelian varieties}

The fundamental example of a Shimura variety is the moduli space of principally polarised abelian varieties of dimension~\( g \).
We recall briefly that a \defterm{polarisation} of an abelian variety of \( A \) is an isogeny \( A \to A^\vee \) to the dual variety satisfying a certain positivity condition.
Any polarisation induces a symplectic form \( H_1(A, \Z) \times H_1(A, \Z) \to \Z \).
A polarisation is \defterm{principal} if it has degree \( 1 \) -- that is if it is an isomorphism \( A \to A^\vee \).

The moduli space of principally polarised abelian varieties of dimension~\( g \) is associated with the Shimura datum \( (\GSp_{2g}, \Hg^{\pm}) \) where \( \GSp_{2g} \) is the group of symplectic similitudes and the conjugacy class of Hodge parameters can be identified with the union of the Siegel upper and lower half spaces
\begin{align*}
   \Hg   & = \{ Z \in \rM_g(\C) \mid Z \text{ is symmetric and } {\Im Z} \text{ is positive definite} \},
\\ \Hg^- & = \{ Z \in \rM_g(\C) \mid Z \text{ is symmetric and } {\Im Z} \text{ is negative definite} \}.\end{align*}
The action of \( \GSp_{2g}(\R) \) on \( \Hg^\pm \) is given by
\[ \fullsmallmatrix{A}{B}{C}{D}Z = (AZ + B)(CZ + D)^{-1}, \quad A, B, C, D \in \rM_g(\R). \]

Taking \( K = \GSp_{2g}(\hat{\Z}) \), we get the Shimura variety \( \Ag \) whose \( \C \)-points are in bijection with isomorphism classes of principally polarised abelian varieties of dimension~\( g \) over \( \C \).

\subsection{Hecke correspondences}

Let \( (G, X) \) be a Shimura datum and \( K \subset G(\A_f) \) a compact open subgroup.
Choose \( g \in G(\A_f) \) and let \( K_g = K \cap gKg^{-1} \).
The inclusion \( K_g \hookrightarrow K \) induces a finite morphism
\[ \pi \colon \Sh_{K_g}(G, X) \to \Sh_K(G, X). \]
Because \( K_g \) is normalised by \( g \), we also have an automorphism \( \tau_g \colon \Sh_{K_g}(G, X) \to \Sh_{K_g}(G, X) \) which sends the double coset \( [x, \theta K_g] \) to \( [x, \theta gK_g] \).
We define a finite correspondence \( T_g \) on \( \Sh_K(G, X) \) by
\[ \xymatrix{
   & \Sh_{K_g}(G, X)	\ar[ld]_{\pi}	\ar[rd]^{\pi \circ \tau_g}
   &
\\
     \Sh_K(G, X)	\ar[rr]^{T_g}
   &
   & \Sh_K(G, X)
} \]

Such correspondences are called \defterm{Hecke correspondences} and, for any point \( s \in \Sh_K(G, X) \), the \defterm{Hecke orbit} of \( s \) is the union
\[ \bigcup_{g \in G(\A_f)} T_g.s. \]

The Hecke correspondence \( T_g \) depends only on the double coset \( KgK \).
If we consider only Hecke correspondences which preserve the neutral component of \( \Sh_K(G, X) \), then each double coset contains an element of \( G(\Q)_+ \).

In the case of the moduli space of principally polarised abelian varieties,
whenever two points \( s, t \in \Ag \) lie in the same Hecke orbit, the associated abelian varieties are isogenous.
The converse is not true: \( s \) and \( t \) are in the same Hecke orbit if and only if there is a \defterm{polarised isogeny} between the associated principally polarised abelian varieties -- that is, an isogeny \( f \colon A_s \to A_t \) such that
\[ f^* \lambda_t = n.\lambda_t \]
for some \( n \in \Z \).

\subsection{Special and weakly special subvarieties} \label{subsec:special-subvars}

Let \( (G, X) \) be a Shimura datum and \( K \subset G(\A_f) \) a compact open subgroup.
A Shimura datum \( (H, X_H) \) is a \defterm{Shimura sub-datum} of \( (G, X) \) if \( H \subset G \) and \( X_H \subset X \).
Letting \( K_H = K \cap H(\A_f) \), we get a finite morphism \( \iota_H \colon \Sh_{K_H}(H, X_H) \to \Sh_K(G, X) \).

An irreducible subvariety of \( S \subset \Sh_K(G, X) \) is said to be a \defterm{special subvariety} if there exist a Shimura sub-datum \( (H, X_H) \subset (G, X) \) and an element \( g \in G(\A_f) \) such that \( S \) is an irreducible component of the image of
\[ T_g \circ \iota_H \colon \Sh_{K_H}(H, X_H) \to \Sh_K(G, X). \]
The Hecke correspondence \( T_g \) is needed in this definition only because the image of \( \iota_H \) might not intersect every geometrically connected component of \( \Sh_K(G, X) \).
In the case of a geometrically connected Shimura variety such as \( \Ag \), every special subvariety is simply an irreducible component of the image of \( \iota_H \) for some Shimura sub-datum.

A \defterm{Shimura morphism} of Shimura varieties is a morphism induced by a homomorphism of the underlying algebraic groups.
Let \( S = \Sh_{K}(G, X) \) and \( S_1 = \Sh_{K_1}(H_1, X_1) \), and let \( f \colon S \to S_1 \) be a surjective Shimura morphism.
Then the adjoint Shimura datum \( (G^\ad, X^\ad) \) splits as a direct product \( (H_1^\ad \times H_2, X_1^\ad \times X_2) \) for some adjoint semisimple group \( H_2 \).
If the compact open subgroup \( K^\ad \subset G^\ad(\A_f) \) splits as a direct product \( K_1^\ad \times K_2 \), then \( S \) is the union of some connected components of \( S_1^\ad \times \Sh_{K_2}(H_2, X_2) \).
(We can always obtain such a decomposition of \( K^\ad \) by replacing \( K \) by a subgroup of finite index.)

An irreducible subvariety \( Z \subset \Sh_K(G, X) \) is a \defterm{weakly special subvariety} if there exist a sub-Shimura datum \( (H, X_H) \subset (G, X) \), an element \( g \in G(\A_f) \), some other Shimura datum \( (H', X') \) and a Shimura morphism
\[ f \colon \Sh_{K_H}(H, X_H) \to \Sh_{K_1}(H_1, X_1) \]
such that \( Z \) is an irreducible component of the image under \( T_g \circ \iota_H \) of the fibre \( f^{-1}(s) \) for some point \( s \in \Sh_{K_1}(H_1, X_1) \).

An example of a weakly special subvariety is the subvariety of \( \cA_2 \) parameterising principally polarised abelian surfaces of the form \( E_0 \times E \), where \( E_0 \) is a fixed elliptic curve and \( E \) a varying elliptic curve.
This is a special subvariety if and only if \( E_0 \) has complex multiplication.
Here the Shimura variety \( \Sh_{K_H}(H, X_H) \) is \( \cA_1 \times \cA_1 \) and the Shimura morphism \( f \) is the projection onto the first factor.
Note that in this example, \( \iota_H \colon \cA_1 \times \cA_1 \to \cA_2 \) is not injective but is the quotient by the action of \( \Z/2 \) exchanging the two factors \( \cA_1 \).

By the above discussion of the structure of Shimura morphisms, every weakly special subvariety arises from a decomposition of \( (H^\ad, X_H^\ad) \) as a direct product.
In particular, the definition of weakly special subvarieties in terms of fibres of Shimura morphisms is equivalent to the definition in the introduction in terms of a decomposition of an adjoint Shimura datum.

However, not every weakly special subvariety of \( \Ag \) arises from a product decomposition of the associated abelian varieties as in the example.
This is important as it prevents us from proving \cref{main-conj}.
The smallest example of a weakly special subvariety of \( \Ag \) which does not come from a product decomposition of abelian varieties is for \( g = 8 \).
It is due to André~\cite{andre:weakly-special} and Borovoi and is described by Moonen~\cite{moonen:linearity-I}.

\subsection{The Zilber--Pink conjecture}

Let \( S \) be a Shimura variety.
If \( Z \) is a subvariety of \( S \), let \( S_Z \) be the smallest special subvariety of \( S \) containing \( Z \).
This exists because every connected component of an intersection of special subvarieties is special.
We define the \defterm{defect} of \( Z \) to be \( \dim S_Z - \dim Z \).

\begin{conjecture}[\cite{pink:zp}~Conjecture~1.1] \label{zp-conj}
Let \( S \) be a Shimura variety and \( Z \subset S \) an irreducible subvariety over \( \C \).
Suppose that \( Z \) contains a Zariski dense set of points of defect at most \( d \).
Then the defect of \( Z \) itself is at most \( d \).
\end{conjecture}

If we take \( d = 0 \) then this becomes the André--Oort conjecture: an irreducible subvariety of a Shimura variety containing a Zariski dense set of special points is a special subvariety.

\begin{lemma}
\Cref{zp-conj} implies \cref{main-conj}.
\end{lemma}

\begin{proof}
This proof is based on \cite{pink:zp} Theorem~3.3.
Since \cref{main-conj} concerns isogeny classes rather than Hecke orbits, this lemma is slightly stronger than the \( \Ag \) case of Pink's theorem.

The only addition we need to make to Pink's proof is to note that if \( s, t \in \Ag \) correspond to isogenous abelian varieties, then \( s \in \Ag \) and \( (s, t) \in \Ag \times \Ag \) each have the same defect.
This follows from the fact that the abelian varieties \( A_s \) and \( A_s \times A_t \) have isomorphic Mumford--Tate groups.
Specifically, if \( M \subset \GSp_{2g}(\Q) \) is the Mumford--Tate group of \( A_s \) and \( g \in \GL_{2g}(\Q) \) is the rational representation of some isogeny \( A_t \to A_s \) (as in section~\ref{subsec:rat-reps}),
then the Mumford--Tate group of \( A_s \times A_t \) is the image of the embedding \( M \to \GL_{2g} \times \GL_{2g} \) given by \( x \mapsto (x, gxg^{-1}) \).

Let \( S \) be the smallest special subvariety of \( \Ag \) containing the point \( s \), and let \( d = \dim S \).
Let \( S' \subset \Ag \times \Ag \) be the smallest special subvariety of \( \Ag \times \Ag \) containing \( \{ s \} \times Z \).

As we noted above, each point of \( \{s\} \times \Lambda \) has defect \( d \) so by \cref{zp-conj},
\[ \dim S' - \dim Z \leq d. \]

The projection of \( S' \) onto the first factor is a special subvariety of \( \Ag \) containing \( \{s\} \), so it must contain \( S \).
This projection is a Shimura morphism so comes from a product decomposition of the adjoint Shimura datum associated with \( S' \).
Hence all fibres of this projection have the same dimension, which must be at least \( \dim Z \).
So
\[ \dim S' - \dim Z \geq \dim S = d. \]

So we must have equality in both inequalities.
Equality in the latter implies that \( \{ s \} \times Z \) is a fibre of the Shimura morphism \( S' \to S \) and so is weakly special.
This implies that \( Z \) is weakly special in \( \Ag \).
\end{proof}

\section{Proof of main theorem} \label{sec:main-thm}

In this section we will deduce our main theorems \labelcref{main-thm-curves,main-thm-higher} from the matrix height bounds of section~\ref{sec:matrix-heights} and the isogeny bound of section~\ref{sec:isog-fg-fields}.
Accordingly fix a point \( s \in \Ag(\C) \) and let \( \Lambda \) be its isogeny class.
Let \( Z \subset \Ag \) be an irreducible closed algebraic subvariety such that \( Z \cap \Lambda \) is Zariski dense in \( Z \).

We begin with some definitions and notation.
Let \( \pi \colon \Hg \to \Ag \) denote the quotient map and \( \Fg \subset \Hg \) the Siegel fundamental domain.
Let \[ \Ztil = \pi^{-1}(Z) \cap \Fg \quad \text{and} \quad \Ltil = \pi^{-1}(\Lambda) \cap \Fg. \]
Fix a point \( \stil \in \Hg \) such that \( \pi(\stil) = s \).

We define the \defterm{complexity} of a point \( t \in \Lambda \) to be the minimum degree of an isogeny \( A_s \to A_t \) between the abelian varieties corresponding to the points \( s \) and \( t \) of \( \Ag \).
We may also talk about the complexity of a point in \( \Ltil \), meaning the complexity of its image in \( \Lambda \).

For a matrix \( \gamma \in \rM_{n \times n}(\Q) \), the \defterm{height} \( H(\gamma) \) will mean the maximum of the standard multiplicative heights of the entries of \( \gamma \).
A straightforward calculation shows that if \( \gamma_1, \gamma_2 \in \rM_{n \times n}(\Q) \) then
\[ H(\gamma_1 \gamma_2) \leq nH(\gamma_1)H(\gamma_2). \]

\subsection{O-minimality and definability}

A key role in the proof is played by the Pila--Wilkie theorem on definable sets in o-minimal structures.
For an introduction to o-minimality and the Pila--Wilkie theorem, see~\cite{scanlon:bams-survey}.
Here we recall some of the definitions and a strengthened version of the Pila--Wilkie theorem, due to Pila.

A \defterm{structure \( \mathcal{S} \) (over \( \R \))} is a sequence \( \mathcal{S}_n \) of collections of subsets of \( \R^n \) for each natural number \( n \) such that

\begin{enumerate}
\item \( \mathcal{S}_n \) is closed under finite unions, intersections and complements;
\item \( \mathcal{S}_n \) contains all semialgebraic subsets of \( \R^n \) (that is, those sets definable by polynomial inequalities);
\item if \( A \in \mathcal{S}_m \) and \( B \in \mathcal{S}_n \) then \( A \times B \in \mathcal{S}_{m+n} \);
\item if \( m \geq n \) and \( A \in \mathcal{S}_m \) then \( \pi(A) \in \mathcal{S}_n \), where \( \pi \colon \R^m \to \R^n \) is projection onto the first \( n \) coordinates.
\end{enumerate}

The sets in \( \mathcal{S}_n \) are called the \defterm{definable sets} of the structure \( \mathcal{S} \).
A function \( f \colon A \to B \) for \( A \subset \R^m \), \( B \subset \R^n \) is said to be \defterm{definable} if its graph is a definable subset of \( \R^{m+n} \).

A structure is \defterm{o-minimal} if every set in \( \mathcal{S}_1 \) is a finite union of points and intervals.
The basic example of an o-minimal structure is the structure of semialgebraic sets.
The o-minimality condition implies that all definable sets in the structure are topologically well-behaved:
for example they have finitely many connected components, a finite a cell decomposition and can be stratified and triangulated (see \cite{vandendries:tame-topology}).

For the purposes of this article, we will only use the structure \( \R_{\mathrm{an,exp}} \) generated by the graphs of restricted analytic functions and the real exponential function.
Accordingly, \defterm{definable} will henceforth mean definable in \( \R_{\mathrm{an,exp}} \).
A restricted analytic function is a function \( f \colon [0,1]^n \to \R \) which extends to a real analytic function on some open neighbourhood of \( [0,1]^n \).
This structure was shown to be o-minimal by van den Dries and Miller~\cite{vandendries-miller}.

According to \cite{baily:compactification-ag}, there is a map \( \pi \colon \Hg \to \bP^N(\C) \) which induces an embedding of \( \Ag \) as a quasi-projective variety.
Siegel~\cite{siegel:symplectic-geometry} used Minkowski's reduction theory to construct a semialgebraic fundamental domain for the action of \( \Sp_{2g}(\Z) \) on \( \Hg \); we will call this domain \( \Fg \).
Crucially for us, Peterzil and Starchenko have shown that the restriction of \( \pi \) to \( \Fg \) is definable in \( \R_{\mathrm{an,exp}} \) \cite{peterzil-s:definability}.
This implies that, if \( Z \subset \Ag \) is an algebraic subvariety, then \( \pi^{-1}(Z) \cap \Fg \) is definable.

The principal theorem we shall use about o-minimal structures is a strong version of the Pila--Wilkie theorem, which uses definable blocks.
A definable block is a definable set which is connected and almost semialgebraic.
More precisely, a \defterm{(definable) block} of dimension \( w \) in \( \R^n \) is a connected definable subset \( W \subseteq \R^n \) of dimension \( w \), regular at every point, such that there is a semialgebraic set \( A \subseteq \R^n \) of dimension~\( w \), regular at every point, with \( W \subseteq A \).

A \defterm{definable block family} is a definable subset \( W \) of \( \R^n \times \R^m \) such that for each \( \eta \in \R^m \), \( W_\eta = \{ x \in \R^n \mid (x, \eta) \in W \} \) is a definable block.

The statement of the following theorem has been simplified by referring only to rational points and requiring \( Z \) to be a single definable set.
Pila's theorem works for points over number fields of some chosen degree and allows \( Z \) itself to be a definable family but these are not necessary for our purposes.

\begin{theorem} [\cite{pila:aomml}~Theorem~3.6] \label{pila-wilkie}
Let \( Z \subset \R^n \) be a definable set and \( \epsilon > 0 \).
There are a finite number \( J = J(Z, \epsilon) \) of definable block families
\[ \cW^{(j)} \subset \R^n \times \R^m, \; j = 1, \dotsc, J \]
and a constant \( c = c(Z, \epsilon) \) such that:
\begin{enumerate}
\item for all \( \eta \in \R^m \),
\[ \cW^{(j)}_\eta \subset Z; \]
\item for all \( T \geq 1 \), the rational points of \( Z \) of height at most \( T \) are contained in the union of at most \( cT^\epsilon \) definable blocks of the form \( \cW^{(j)}_\eta \) (for some \( j \in \{ 1, \dotsc, J \} \) and some \( \eta \in \R^m \)).
\end{enumerate}
\end{theorem}

\subsection{Outline of proof}

The key step in the proof of the main theorems is \cref{blocks-symm-space}: the points of \( \Ztil \cap \Ltil \) of a given complexity are contained in subpolynomially many definable blocks, these blocks themselves contained in \( \Ztil \).
This is proved using \cref{pila-wilkie} (Pila's theorem) and the matrix height bounds of section~\ref{sec:matrix-heights}.

\begin{proposition} \label{blocks-symm-space}
Let \( Z \) be a subvariety of \( \Ag \) and \( \tilde{s} \) a point in \( \Hg \).
Let \( \epsilon > 0 \).

There is a constant \( c = c(Z, \tilde{s}, \epsilon) \) such that for every \( n \geq 1 \),
there is a collection of at most \( c n^\epsilon \) definable blocks \( W_i \subset \Ztil \) such that
the union \( \bigcup W_i \) contains all points of \( \Ztil \cap \Ltil \) of complexity~\( n \).
\end{proposition}

On the other hand, the Masser--Wüstholz isogeny theorem gives a polynomial lower bound for the Galois degree of points in \( \Lambda \) in terms of their complexity.
Combining these two bounds, once the complexity gets large enough there are more points in \( \Ztil \cap \Ltil \) than there are blocks to contain them.
Hence most points of \( \Ztil \cap \Ltil \) are contained in blocks of positive dimension.
In particular the union of positive-dimensional blocks contained in \( \Ztil \) has Zariski dense image in~\( Z \).

In the case \( \dim Z = 1 \) this implies that \( \Ztil \) has an algebraic irreducible component and so
we can conclude using the Ullmo--Yafaev characterisation of weakly special subvarieties.
When \( \dim Z > 1 \), we use the Ax--Lindemann theorem for \( \Ag \) to deduce that positive-dimensional weakly special subvarieties are dense in \( Z \) and then a result of Ullmo to complete the proof of \cref{main-thm-higher}.

\medskip

Let us outline the proof of \cref{blocks-symm-space}.
We cannot apply the counting theorem to \( \Ltil \subset \Ztil \) directly, because the points of \( \Ltil \) are transcendental.
Instead we construct a definable set \( Y \) and a semialgebraic map \( \sigma \colon Y \to \Ztil \) such that points of \( \Ztil \cap \Ltil \) have rational preimages in \( Y \), with heights polynomially bounded in terms of their complexity.
This idea is due to Habegger and Pila~\cite{habegger-pila:unlikely}.

Consider first the case \( \End A_s = \Z \).
This case is easier because all isogenies between \( A_s \) and any abelian variety are polarised.
In this case we let
\[ Y = \{ \gamma \in \GSp_{2g}(\R)^+ \mid \gamma.\stil \in \Ztil \}, \]
and let \( \sigma \colon Y \to \Ztil \) be the map \( \sigma(\gamma) = \gamma.\stil \).

Let \( \ttil \in \Ztil \cap \Ltil \) and \( t = \pi(\ttil) \).
Then there is an isogeny \( f \colon A_t \to A_s \) whose degree is equal to the complexity of~\( t \).
By the hypothesis \( \End A_s = \Z \) this isogeny is polarised.
Hence the rational representation of~\( f \) (explained in section~\ref{sec:matrix-heights}) gives a matrix \( \gamma \in \GSp_{2g}(\Q)^+ \) such that \( \pi(\gamma.\stil) = t \) and whose height is polynomially bounded with respect to the complexity.
We can also find \( \gamma_1 \in \Sp_{2g}(\Z) \) of polynomially bounded height such that \( \gamma_1 \gamma.\stil = \ttil \).
Hence every point in \( \Ztil \cap \Ltil \) has a rational preimage in \( Y \) of polynomially bounded height.
This is precisely what we need to apply \cref{pila-wilkie} to~\( Y \).

\medskip

If we drop the assumption \( \End A_s = \Z \) then this no longer works, because the rational representation of a non-polarised isogeny is not in \( \GSp_{2g}(\Q)^+ \).
Note that even if we assume that \( t \) is in the Hecke orbit of \( s \), so that there is some polarised isogeny \( A_s \to A_t \), the isogeny of minimum degree need not be polarised.
Thus we do not get an element of \( \GSp_{2g}(\Q)^+ \) whose height is polynomially bounded in terms of the complexity.

To avoid this problem we will take \( Y \) to be a subset of \( \GL_{2g}(\R) \) instead of \( \GSp_{2g}(\R) \).
This will allow us to carry out the same proof using the rational representation of a not-necessarily-polarised isogeny.
Of course \( \GL_{2g}(\R) \) does not act on \( \Hg \) but this does not matter:
the map
\[ \sigma\fullsmallmatrix{A}{B}{C}{D} = (A\stil + B)(C\stil + D)^{-1} \quad \text{for } A, B, C, D \in \rM_{g \times g}(\R) \]
is defined on a Zariski open subset of \( \GL_{2g}(\R) \), and we will only consider matrices in \( \GL_{2g}(\R) \) where \( \sigma \) is defined and has image in \( \Hg \).
In particular let
\[ Y = \sigma^{-1}(\Ztil). \]

\subsection{Proof of Proposition~\ref{blocks-symm-space}}

Before proving \cref{blocks-symm-space}, we need to check that every element \( t \in \Ztil \cap \Ltil \) has a rational preimage in \( Y \) whose height is polynomially bounded with respect to the complexity.
\Cref{rat-rep-height-bound} says that \( t \) has some preimage in \( \GL_{2g}(\R) \) with this property, but we need to move this into the preimage of the fundamental domain~\( \Fg \).

In the published version of this paper, \cref{matrix-height-bound-fg} is deduced directly from \cref{rat-rep-height-bound} and \cite{pila-tsimerman:ao-surfaces}~Lemma~3.2.
During his PhD studies, Gabriel Dill realised that the height bound given by \cref{rat-rep-height-bound} is not sufficient for this use of \cite{pila-tsimerman:ao-surfaces}~Lemma~3.2, because the latter requires bounds on the period matrix.
Dill gave a proof for the required bounds on the period matrix, which has been added to the arXiv version of this paper as \cref{correction:rat-rep-height-bound-plus}.

\begin{lemma} \label{matrix-height-bound-fg}
There exist constants \( c, k \) depending only on \( g \) and \( \stil \) such that:

For any \( \ttil \in \Ztil \cap \Ltil \) of complexity~\( n \), there is a rational matrix \( \gamma \in Y \) such that \( \sigma(\gamma) = \ttil \) and \( H(\gamma) \leq cn^k \).
\end{lemma}

\begin{proof}
Let \( t = \pi(\ttil) \).
Let \( \cB \) be a symplectic basis for \( H_1(A_s, \Z) \) with period matrix~\( \stil \).

By \cref{correction:rat-rep-height-bound-plus}, there are an isogeny \( f \colon A_s \to A_t \) and a symplectic basis~\( \cB' \) for \( H_1(A_t, \Z) \) such that the rational representation \( \gamma_1 \) of \( f \) has polynomially bounded height.

Let \( \util \) denote the period matrix for \( (A_t, \lambda_t) \) with respect to the basis~\( \cB' \).
As remarked in section~\ref{subsec:rat-reps}, \( \stil = \tp{\gamma_1}.\util \) or in other words
\[ \util = \sigma(\tp{\gamma_1^{-1}}). \]

Because \( \util \) is a period matrix for \( (A_t, \lambda_t) \), there exists \( \gamma_2 \in \Sp_{2g}(\Z) \) such that \( \gamma_2.\util = \ttil \).
According to \cref{correction:rat-rep-height-bound-plus}, \( \max(\abs{\util_{ij}}, (\det \Im \util)^{-1}) \) is bounded by a polynomial in~\( n \).
We can therefore apply \cite{pila-tsimerman:ao-surfaces}~Lemma~3.2 to conclude that the height of \( \gamma_2 \) is polynomially bounded.

Thus \( \gamma = \gamma_2 \, \tp{\gamma_1^{-1}} \) satisfies the required conditions.
\end{proof}

Now we are ready to prove \cref{blocks-symm-space}.
We simply apply \cref{pila-wilkie} to~\( Y \), using \cref{matrix-height-bound-fg} to relate heights of rational points in~\( Y \) to complexities of points in \( \Ztil \cap \Ltil \).
We then use the fact that \( \sigma \) is semialgebraic, and that the blocks in \( Y \) can be chosen uniformly from finitely many definable families, to go from \( Y \) to \( \Ztil \).

\begin{proof}[Proof of \cref{blocks-symm-space}]
The set
\[ Y = \sigma^{-1}(\pi^{-1}(Z) \cap \Fg) \]
is definable because \( \sigma \) is semialgebraic and \( \pi_{|\Fg} \) is definable by a theorem of Peterzil and Starchenko~\cite{peterzil-s:definability}.

Hence we can apply \cref{pila-wilkie} to \( Y \):
for every \( \epsilon > 0 \), there are finitely many definable block families \( \cW^{(j)}(\epsilon) \subset Y \times \R^m \) and a constant \( c_1(Y, \epsilon) \) such that for every \( T \geq 1 \),
the rational points of \( Y \) of height at most \( T \) are contained in the union of at most \( c_1 T^\epsilon \) definable blocks~\( W_i(T, \epsilon) \), taken from the families \( \cW^{(j)}(\epsilon) \).

Since \( \sigma \) is semialgebraic, the image under \( \sigma \) of a definable block in \( Y \) is a finite union of definable blocks in \( \Ztil \).
Furthermore the number of blocks in the image is uniformly bounded in each definable block family~\( \cW^{(j)}(\epsilon) \).
Hence \( \sigma(\bigcup W_i(T, \epsilon)) \)
is the union of at most \( c_2 T^\epsilon \) blocks in \( \Ztil \),
for some new constant \( c_2(Z, \stil, \epsilon) \).

But by \cref{matrix-height-bound-fg}, for suitable constants \( c, k \), every point of \( \Ztil \cap \Ltil \) of complexity~\( n \) is in \( \sigma(\bigcup W_i(cn^k, \epsilon)) \).
\end{proof}

\subsection{End of proof of Theorems~\ref{main-thm-curves} and~\ref{main-thm-higher}}

\Cref{blocks-symm-space} tells us that the points of \( \Ztil \cap \Ltil \) of complexity \( n \) are contained in fewer than \( c(\epsilon)n^\epsilon \) blocks for every \( \epsilon > 0 \).
On the other hand, the Masser--Wüstholz isogeny theorem implies that the number of such points grows at least as fast as~\( n^{1/k} \) for some constant~\( k \).
Hence most points of \( \Ztil \cap \Ltil \) are contained in a block of positive dimension.
We check that this is sufficient, with the hypothesis that \( Z \cap \Lambda \) is Zariski dense in \( Z \), to deduce that the union of positive-dimensional blocks in \( \Ztil \) has Zariski dense image in \( Z \).

\begin{proposition} \label{blocks-dense}
Let \( \Lambda_1 \) be the set of points \( t \in Z \cap \Lambda \) for which there is a positive-dimensional block \( W \subset \Ztil \) such that \( t \in \pi(W) \).

If \( Z \cap \Lambda \) is Zariski dense in \( Z \), then \( \Lambda_1 \) is Zariski dense in \( Z \).
\end{proposition}

\begin{proof}
Let \( Z_1 \) denote the Zariski closure of \( \Lambda_1 \) (we do not yet know that this is non-empty).

Let \( (A_s, \lambda_s) \) be a polarised abelian variety corresponding to the point \( s \in \Ag(\C) \), defined over a finitely generated field \( K \).
We choose \( K \) large enough that the varieties \( Z \) and~\( Z_1 \) are also defined over \( K \).

Let \( t \) be a point in \( Z \cap \Lambda \) of complexity \( n \).
The polarised abelian variety corresponding to \( t \) might not have a model over the field of moduli \( K(t) \), but it has a model \( (A_t, \lambda_t) \) over an extension \( L \) of \( K(t) \) of uniformly bounded degree.
This follows from the fact that a polarised abelian variety with full level-\( 3 \) structure has no non-trivial automorphisms (\cite{milne:abelian-varieties-old} Proposition~17.5), so is defined over its field of moduli;
and the field of moduli of a full level-\( 3 \) structure on the polarised abelian variety corresponding to \( t \) is an extension of \( K(t) \) of degree at most \( \lvert \Sp_{2g}(\Z/3) \rvert \).

By \cref{isogeny-bound},
the complexity~\( n \) is bounded above by a polynomial \( c[L:K]^k \) in \( [L:K] \), with \( c \) and~\( k \) depending only on \( A_s \) and \( K \).
Hence for a different constant \( c_1 \), we have
\[ [K(t):K] \geq c_1 n^{1/k}. \]

But all \( \Gal(\bar{K}/K) \)-conjugates of \( t \) are contained in \( Z \cap \Lambda \) and have complexity~\( n \).
By \cref{blocks-symm-space}, the preimages in \( \Fg \) of these points are contained in the union of at most \( c_2(Z, \tilde{s}, 1/2k) n^{1/2k} \) definable blocks, each of these blocks being contained in \( \Ztil \).

For large enough \( n \), we have
\[ c_1 n^{1/k} > c_2 n^{1/2k}. \]
For such \( n \), by the pigeonhole principle there is a definable block \( W \subset \Ztil \) such that \( \pi(W) \) contains at least two Galois conjugates of \( t \).
Since blocks are connected by definition, \( \dim W > 0 \).
So those conjugates of \( t \) in \( \pi(W) \) are in \( \Lambda_1 \).
Since \( Z_1 \) is defined over \( K \), it follows that \( t \) itself is also in \( Z_1 \).

In other words all points of \( Z \cap \Lambda \) of large enough complexity are in \( Z_1 \).
But this excludes only finitely many points of \( Z \cap \Lambda \).
So as \( Z \cap \Lambda \) is Zariski dense in \( Z \), we conclude that \( Z_1 = Z \).
\end{proof}

Call a subset \( W \subset \Hg \) \defterm{complex algebraic} if it is a connected component of \( W_0 \cap \Hg \) for some irreducible complex algebraic variety \( W_0 \subset \rM_{2g \times 2g}(\C) \).
Let \( \Ztil^\ca \) denote the complex algebraic part of~\( \Ztil \) --  that is, the union of positive-dimensional complex algebraic subsets of \( \Hg \) contained in \( \Ztil \).

By Lemma~2.1 of~\cite{pila:alg-to-ca}, \( \Ztil^\ca \) is the same as the union of the definable blocks contained in \( \Ztil \).
So \cref{blocks-dense} tells us that \( \pi(\Ztil^\ca) \) is Zariski dense in \( Z \).

If \( \dim Z = 1 \), then the fact that \( \Ztil^\ca \) is non-empty implies that some irreducible component of \( \Ztil \) is complex algebraic.
By~\cite{uy:characterisation} this implies that \( Z \) is weakly special, proving \cref{main-thm-curves}.

For \( \dim Z > 1 \), we use the Ax--Lindemann theorem for \( \Ag \) proved by Pila and Tsimerman~\cite{pila-tsimerman:ax-lindemann}: if \( W \) is a maximal complex algebraic subset of \( \Ztil \) then \( \pi(W) \) is weakly special.
Hence \( \pi(\Ztil^\ca) \) is a union of positive-dimensional weakly special subvarieties, so these are dense in \( Z \).
Let \( S \) be the smallest special subvariety of \( \Ag \) containing \( Z \).
By Théorème~1.3 of~\cite{ullmo:hyp-ax-lind}, we deduce that \( S = S_1 \times S_2 \) for some Shimura varieties \( S_1 \) and \( S_2 \), and \( Z = S_1 \times Z' \) for some subvariety \( Z \subset S_2 \), proving \cref{main-thm-higher}.

\section{Heights of rational representations of isogenies} \label{sec:matrix-heights}

Let \( (A, \lambda) \) and \( (A', \lambda') \) be principally polarised abelian varieties over \( \C \) related by an isogeny of degree~\( n \) (not necessarily compatible with the polarisations).
In this section we show that, for suitable choices of bases for \( H_1(A, \Z) \) and \( H_1(A', \Z) \) and of isogeny \( f \) between \( A \) and~\( A' \), the height of the rational representation of~\( f \) is polynomially bounded in \( n \).
A precise statement of this bound is given at \cref{rat-rep-height-bound}.
This is derived from \cref{heights-of-endoms}, which gives a height bound for endomorphisms of \( A \).

The height bounds from this section are used in the proof of \cref{matrix-height-bound-fg}.
Gabriel Dill discovered that \cref{rat-rep-height-bound} is not sufficient for this purpose.
The proof of \cref{matrix-height-bound-fg} also requires a bound for the period matrix of the basis of \( H_1(A', \Z) \) used in the proposition.
Dill proved such a bound by using the full strength of \cref{heights-of-endoms}.
A modified version of \cref{rat-rep-height-bound} which includes this bound for the period matrix can be found at \cref{correction:rat-rep-height-bound-plus} (\cref{correction:rat-rep-height-bound-plus} is only in the arXiv version of the paper as it was added after the paper was published).

The notation \( H(f, \cB', \cB) \) in the \lcnamecref{rat-rep-height-bound} refers to the height of the rational representation of the isogeny \( f \) with respect to bases \( \cB' \), \( \cB \) of the period lattices.
This is defined below in section~\ref{subsec:rat-reps}.

\begin{proposition} \label{rat-rep-height-bound}
Let \( (A, \lambda) \) be a principally polarised abelian variety over \( \C \) and fix a symplectic basis \( \cB \) for \( H_1(A, \Z) \).
There exist constants \( c, k \) depending only on \( (A, \lambda) \) such that:

If \( (A', \lambda') \) is any principally polarised abelian variety for which there exists an isogeny \( A \to A' \) of degree~\( n \), then
there are an isogeny \( f \colon A' \to A \) and a symplectic basis \( \cB' \) for \( H_1(A', \Z) \) such that
\[ H(f, \cB', \cB) \leq cn^k. \]
\end{proposition}

In this \lcnamecref{rat-rep-height-bound}, the isogeny whose existence is assumed and the isogeny whose existence is asserted in the conclusion go in opposite directions.
This is the most convenient formulation for our application, but it is not important since any isogeny \( A \to A' \) of degree~\( n \) gives rise to an isogeny in the opposite direction of degree~\( n^{2g-1} \).

\subsection{Polarisations}

Let \( A \) be an abelian variety and \( A^\vee \) its dual variety.
A \defterm{polarisation} of \( A \) is an isogeny \( A \to A^\vee \) satisfying a certain positivity condition (being associated with an ample sheaf).
Any polarisation induces a symplectic form \( H_1(A, \Z) \times H_1(A, \Z) \to \Z \).
A polarisation is \defterm{principal} if it is an isomorphism \( A \to A^\vee \).

A polarisation \( \lambda \colon A \to A^\vee \) induces an involution, called the \defterm{Rosati involution}, of \( \End A \otimes_\Z \Q \) defined by
\[ a^\dag = \lambda^{-1} \circ a^\vee \circ \lambda \]
where \( a^\vee \) means the morphism dual to \( a \).
This involution reverses the order of multiplication in \( \End A \otimes_\Z \Q \).
It gives an involution of \( \End A \) itself if \( \lambda \) is principal.

Having fixed a principal polarisation \( \lambda \) of \( A \), every other polarisation has the form \( \lambda \circ q \) for some \( q \in \End A \) which is \defterm{symmetric}, i.e.\ \( q^\dag = q \), and \defterm{positive definite}, i.e.\ each component of \( q \) in
\[ \End A \otimes_\Z \R \cong \prod \rM_{l_i}(\R) \times \prod \rM_{m_i}(\C) \times \prod \rM_{n_i}(\HH) \]
has eigenvalues which are positive real numbers.

\subsection{Rational representations} \label{subsec:rat-reps}

We define the \defterm{rational representation} of an isogeny \( f \colon A' \to A \) (with respect to bases \( \cB, \cB' \) for \( H_1(A, \Z) \) and \( H_1(A', \Z) \)) to be the matrix of the induced morphism
\[ f_* \colon H_1(A', \Z) \to H_1(A, \Z) \]
in terms of the chosen bases.
This gives a \( 2g \times 2g \) integer matrix.
We write
\[ H(f, \cB', \cB) \]
for the \defterm{height} of the rational representation of \( f \), meaning simply the maximum of the absolute values of the entries of the matrix.

Rational representations of isogenies are particularly interesting in the case that the bases \( \cB, \cB' \) are symplectic with respect to the polarisations \( \lambda, \lambda' \).
In this case, if \( \stil, \ttil \in \Hg \) are the period matrices of \( (A, \lambda) \) and \( (A', \lambda') \) with respect to the chosen bases and \( \gamma \) is the rational representation of an isogeny \( A' \to A \), then
\[ \ttil = (A\stil + B)(C\stil + D)^{-1} \text{ where } \tp{\gamma} = \fullsmallmatrix{A}{B}{C}{D}. \]
We remark also that for symplectic bases, an isogeny is polarised if and only if its rational representation is in \( \GSp_{2g}(\Q) \).

\subsection{Outline of proof}

In the situation of \cref{rat-rep-height-bound}, let \( h \colon A \to A' \) be an isogeny of degree~\( n \).
Then \( h^* \lambda' \) is a polarisation of \( A \), so there is a symmetric positive definite endomorphism \( q \in \End A \) such that
\[ h^* \lambda' = \lambda \circ q. \]

We can identify \( H_1(A', \Z) \) with a submodule of \( H_1(A, \Z) \) of index \( n^{2g-1} \), and so find a basis for \( H_1(A', \Z) \) whose height is at most \( n^{2g-1} \).
However this need not be a symplectic basis.
We apply the standard algorithm for finding a symplectic basis: the height of this new basis is controlled by \( h^* \lambda' \), in other words by \( q \).

So we would like to bound the height of the rational representation of \( q \) in terms of \( \deg h \).
However this is not possible: let \( A \) be an abelian variety whose endomorphism ring is the ring of integers \( \fo \) of a real quadratic field.  In particular \( \fo \) has infinitely many units.
Let \( h \) be a unit in \( \fo \) -- in other words, an isomorphism \( A \to A \).
If we take the same polarisation on each copy of \( A \), then \( q = h^2 \) and the rational representation of this can have arbitrarily large height.

We can avoid this by replacing \( h \) by \( h \circ u \) for some automorphism \( u \) of \( A \) -- recall that all we have supposed about \( h \) is that it is an isogeny \( A \to A' \) of degree~\( n \).
This replaces \( q \) by \( u^\dag qu \).
We will show that we can choose \( u \) so that the height of the rational representation of \( u^\dag qu \) is bounded by a multiple of~\( \deg q = n^2 \).

\subsection{Heights in the endomorphism ring}

The following \lcnamecref{heights-of-endoms} is motivated by the theorem (\cite{milne:abelian-varieties-old} Proposition~18.2) that the symmetric elements of \( \End A \) of a given norm fall into finitely many orbits under the action of \( (\End A)^\times \) given by \( (u, q) \mapsto u q u^\dag \).
In geometric terms, this says that if we fix \( A \) and \( \deg \mu \) then there are finitely many isomorphism classes of polarised abelian varieties \( (A, \mu) \).
Our \lcnamecref{heights-of-endoms} strengthens this by saying that each orbit contains an element whose height is bounded by a multiple of the norm.
Milne's theorem is proved using the reduction theory of arithmetic groups.
We also use reduction theory, but in order to get height bounds we have to go deeper into the structure of \( \End A \otimes_\Z \R \).

The representation \( \rho \) appears in the \lcnamecref{heights-of-endoms} solely to give us a convenient definition of heights and norms of elements of \( R \).
Specifically, \( H(x) \) means the height of \( \rho(x) \) and \( \Nm(x) = \det \rho(x) \) for \( x \in R \).

\begin{proposition} \label{heights-of-endoms}
Let \( (E, \dag) \) be a semisimple \( \Q \)-algebra with a positive involution, let \( R \) be a \( \dag \)-stable order in \( E \) and let \( \rho : R \to \rM_N(\Z) \) be a faithful representation of \( R \).

There is a constant \( c \) depending only on \( (R, \dag, \rho) \) such that for any symmetric positive definite \( q \in R \), there is some \( u \in R^\times \) such that
\[ H(u^\dag qu) \leq c \Nm(q). \]
\end{proposition}

\begin{proof}
We begin by checking that it suffices to prove the \lcnamecref{heights-of-endoms} for simple algebras~\( E \).
In general, \( E = \prod E_i \) for some simple \( \Q \)-algebras~\( E_i \).
Let \( R_i = R \cap E_i \).
Then \( R' = \prod R_i \) is an order of \( E \) contained in \( R \).
Let \( m = [R:R'] \).
Given \( q \in R \), we look at \( mq \in R' \).
Suppose that the \lcnamecref{heights-of-endoms} holds for each \( R_i \); then clearly it holds for \( R' \),
so there is \( u \in R'^\times \) (a fortiori \( u \in R^\times \)) such that
\[ H(umqu^\dag) \leq c\Nm(mu). \]
Hence the \lcnamecref{heights-of-endoms} holds for \( R \) with constant \( c\Nm(m)/m \).

So we suppose that \( E \) is simple.
Then \( E = \rM_n(D) \) for some division algebra \( D \), and the involution \( \dag \) is matrix transposition composed with some involution of \( D \).
We may also suppose that \( R \) is contained in the maximal order \( \rM_n(\fo) \), where \( \fo \) is a maximal order in \( D \).

By the Albert classification of division algebras with positive involution, \( E \otimes_\Q \R \) is isomorphic to one of \( \rM_{nd}(\R)^r \), \( \rM_{nd}(\C)^r \) or \( \rM_{nd}(\HH)^r \).
Because \( q \) is symmetric, its projection onto each simple factor of \( E \otimes_\Q \R \) is a Hermitian matrix.
By the theory of Hermitian forms over \( \R \), \( \C \) and \( \HH \),
there exist \( x, d \in E \otimes_\Q \R \) such that \( d \) is diagonal with real entries in each factor and
\[ q = x^\dag d x. \]
Since \( q \) is positive definite, all the diagonal entries of \( d \) are positive so we can multiply each row of \( x \) by the square root of the corresponding entry of \( d \) to suppose that \( d = 1 \).
We then have \( q = x^\dag x \).

Let \( G \) be the \( \Z \)-group scheme representing the functor \( \Z\text{-}\mathbf{Alg} \to \mathbf{Grp} \) given by
\[ G(A) = (R \otimes_\Z A)^\times. \]
Over \( \Q \) this is the reductive group \( \Res_{D/\Q} \GL_n \).
We will use the following notations for subgroups of~\( G \):
\begin{enumerate}[(i)]
\item \( S \) is the maximal \( \Q \)-split torus of \( G \) whose \( \Q \)-points are the diagonal matrices of \( \GL_n(D) \) with entries in \( \Q \);
\item \( P \) is the minimal parabolic \( \Q \)-subgroup of \( G \) consisting of upper triangular matrices;
\item \( U = R_u(P) \) is the group of upper triangular matrices with ones on the diagonal;
\item \( Z \) is the centraliser of \( S \) in \( G \); that is, \( Z(\Q) \) consists of the diagonal matrices in \( \GL_n(D) \);
\item \( M \) is the maximal \( \Q \)-anisotropic subgroup of \( Z \); that is, \( M(\Q) \) consists of the diagonal matrices in \( \GL_n(D) \) whose diagonal entries have reduced norm~\( \pm 1 \);
\item \( K = \{ g \in G(\R) \mid g^\dag g = 1 \} \) is a maximal compact subgroup of \( G(\R) \).
\end{enumerate}

By Proposition~13.1 of~\cite{borel:arithmetic-groups}, there exist a positive real number \( t \), a finite set \( C \subset G(\Q) \) and a compact neighbourhood \( \omega \) of \( 1 \) in \( M^0(\R)U(\R) \) such that
\[ G(\R) = K A_t \omega C G(\Z) \]
where
\[ A_t = \{ a \in S(\R) \mid a_i > 0, a_i/a_{i+1} \leq t \text{ for all } i \}. \]
We note that \( M^0(\R)U(\R) \) is the group of upper triangular matrices in \( \rM_n(D \otimes_\Q \R) \) whose diagonal entries have reduced norm~\( 1 \).

Hence we can write
\[ x = k a z \nu \gamma \]
where \( k \in K \), \( a \in A_t \), \( z \in \omega \), \( \nu \in C \) and \( \gamma \in G(\Z) = R^\times \).

Let \( u = \gamma^{-1} \) and
\[ q' = u^\dag q u. \]
In order to prove the \lcnamecref{heights-of-endoms}, it will suffice to show that \( H(q') \leq c\Nm(q) \).

Since \( k^\dag k = 1 \), and using the decomposition of \( x \), we get that
\[ q' = \nu^\dag z^\dag a^\dag a z \nu. \]
Fix some \( \Z \)-basis of \( R \).
We will show below that the (real) coordinates of \( a^\dag a \) are bounded above by a constant multiple of \( \Nm(q) \).
The coordinates of \( z \) and \( \nu \) are uniformly bounded because \( z \) is in the compact set \( \omega \) and \( \nu \) is in the finite set \( C \).
Hence the coordinates of \( q' \) in this basis are bounded by a multiple of \( \Nm(q) \),
so \( H(q') \) is likewise linearly bounded.

Let \( a^\dag a = \operatorname{diag}(a_1, \dotsc, a_n) \) with \( a_i \in \R \).
In order to show that the coordinates of \( a^\dag a \) in the chosen basis are bounded, it will suffice to show that the \( a_i \) are bounded by a multiple of \( \Nm(q) \).
We shall show that the \( a_i \) are bounded \emph{below} by a constant, and that their product \( \prod a_i \) is bounded above by a multiple of \( \Nm(q) \).
These two facts together imply that the \( a_i \) are bounded above by a multiple of \( \Nm(q) \).

Choose an integer \( m \) such that \( m\nu^{-1} \in R \) for all \( \nu \in C \).
Then
\[ m^2 z^\dag a^\dag a z = (m\nu^{\dag-1}) q' (m\nu^{-1}) \in R \]
so every entry of \( m^2 z^\dag a^\dag a z \), viewed as a matrix in \( \rM_n(D) \), is in \( \fo \).

Let \( z_{11} \) denote the upper left entry of \( z \in \rM_n(D \otimes_\Q \R) \).
Because \( z \) is upper triangular,
the upper left entry of \( m^2 z^\dag a^\dag a z \) is \( m^2 z_{11}^\dag a_1 z_{11} \).
So \( m^2 z_{11}^\dag a_1 z_{11} \in \fo \) and
\[ \abs{\Nrd_{D/\Q}(m^2 z_{11}^\dag a_1 z_{11})} \geq 1. \]
But \( \Nrd(z_{11}) = 1 \) because \( z \in \omega \), so
\[ \abs{\Nrd_{D \otimes_\Q \R/\R}(m^2 a_1)} \geq 1. \]

Since \( m^2 a_1 \) is a positive real number, \( \Nrd_{D \otimes_\Q \R/\R}(m^2 a_1) \) is just some fixed positive power of \( m^2 a_1 \) so we conclude that
\[ m^2 a_1 \geq 1. \]
From the definition of \( A_t \), it follows that \( a_i \geq m^{-2} t^{2-2i} \) for all \( i \) and we have established that the \( a_i \) are uniformly bounded below.

Hence there is a constant \( c_1 \) such that for every \( j \),
\[ a_j \leq c_1 \prod a_i. \]

Since \( \rho \) is faithful \( \dim \rho \geq n \).
Together with the fact that \( \prod a_i \) is bounded below this implies that
\[ \prod a_i \leq c_2 \left( \prod a_i \right)^{\dim \rho/n} = c_2 \Nm(a^\dag a). \]

Now \( \Nm(z) = \Nm(u) = 1 \) and \( \Nm(\nu) \) is bounded because \( \nu \) comes from a finite set,
so \( \Nm(a^\dag a) \) is bounded above by a constant multiple of \( \Nm(q) \).
Combining all this we have proved that each \( a_i \) is bounded above by a constant multiple of \( \Nm(q) \),
and as remarked above this suffices to establish the \lcnamecref{heights-of-endoms}.
\end{proof}

\subsection{Height of a symplectic basis}

We will need the following bound for the height of a symplectic basis for a symplectic free \( \Z \)-module in terms of the values of the symplectic pairing on the standard basis.
The proof is simply to apply the standard recursive algorithm for finding a symplectic basis, verifying that the new vectors introduced always have polynomially bounded heights.

\begin{lemma} \label{symplectic-basis-bound}
Let \( L = \Z^{2g} \) and let \( \{ e_1, \dotsc, e_{2g} \} \) be a basis for \( L \).
There exist constants \( c, k \) depending only on \( g \) such that:

For any perfect symplectic pairing \( \psi : L \times L \to \Z \) with
\[ N = \max_{i, j} \abs{\psi(e_i, e_j)}, \]
there exists a symplectic basis for \( (L, \psi) \) whose coordinates with respect to the basis \( \{ e_1, \dotsc, e_{2g} \} \) are at most \( cN^k \).
\end{lemma}

\begin{proof}
For any \( x \in L \), we write \( H(x) \) for the maximum of the absolute values of the coordinates of \( x \) with respect to the basis \( \{ e_1, \dotsc, e_{2g} \} \).

First let \( e'_1 = e_1 \) and choose \( e'_2 \) such that \( \psi(e'_1, e'_2) = 1 \) and \( H(e'_2) \leq N \).
We can do this because \( \psi \) is perfect, so that \( \gcd_{i=2}^n(\psi(e_1, e_i)) = 1 \).
Hence there are integers~\( a_i \) such that \( \abs{a_i} \leq N \) and
\[ \sum a_i \psi(e_1, e_i) = 1 \]
We let \( e'_2 = \sum a_i e_i \).

Then find \( e'_3, \dotsc, e'_{2g} \) orthogonal to \( e'_1 \) and to \( e'_2 \) such that \( \{ e'_1, \dotsc, e'_{2g} \} \) is a basis for \( L \) and \( H(e'_i) \leq 2gN^2 \).
We can do this by setting
\[ e'_i = e_i + \psi(e'_2, e_i) e'_1 + \psi(e'_1, e_i) e'_2. \]
Here we have \( \abs{\psi(e'_2, e_i)} \leq \sum_{j=2}^n \abs{a_j \psi(e_j, e_i)} \leq (2g-1)N^2 \)
and \( \psi(e'_1, e_i) e'_2 \) has height at most \( N^2 \) so \( H(e'_i) \leq 2gN^2 \).

Finally apply the algorithm recursively to \( L' = \Z\langle e'_3, \dotsc, e'_{2g} \rangle \).
We have
\[  \abs{\psi(e'_i, e'_j)} \leq gNH(e'_i)H(e'_j) \leq 4g^3N^5. \]
Hence by induction \( L' \) has a symplectic basis whose coordinates with respect to \( \{ e'_3, \dotsc, e'_{2g} \} \) are bounded by a constant multiple of \( N^{5k'} \), where \( k' \) is the exponent in the \lcnamecref{symplectic-basis-bound} for \( \Z^{2(g - 1)} \).
Converting these into coordinates with respect to \( \{ e_1, \dotsc, e_{2g} \} \), we get that the elements of this symplectic basis for \( L' \) have height bounded by a constant multiple of \( N^{2+5k'} \).
This proves the \lcnamecref{symplectic-basis-bound}.

We remark that the recurrence \( k(g) = 2+5k(g-1), \, k(0) = 0 \) is satisfied by \( k(g) = (5^g - 1)/2 \), so this provides a suitable choice of exponent for the lemma.
\end{proof}

\subsection{Proof of Proposition~\ref{rat-rep-height-bound}}
Let \( h \colon A \to A' \) be an isogeny of degree~\( n \).
There is \( q \in \End A \) such that
\[ h^* \lambda' = \lambda \circ q. \]
Apply \cref{heights-of-endoms} to get \( u \in (\End A)^\times \) such that
\[ H(u^\dag qu) \leq c \Nm(q). \]
Then \( hu \) is an isogeny \( A \to A' \) of degree~\( n \), so there is also an isogeny \( f \colon A' \to A \) of degree~\( n^{2g-1} \) such that
\[hu \circ f = [n]_{A'}. \]

The image of \( f_* \colon H_1(A', \Z) \to H_1(A, \Z) \) is a submodule of index \( n^{2g-1} \).
By the structure theory of finitely generated \( \Z \)-modules there is a basis \( \{ e'_1, \dotsc, e'_{2g} \} \) for \( H_1(A', \Z) \) with respect to which the rational representation of \( f \) is upper triangular and has height at most \( n^{2g-1} \).
But this need not be a symplectic basis.

Let \( \psi, \psi' \) be the symplectic forms on \( H_1(A, \Z) \) and \( H_1(A', \Z) \) induced by \( \lambda, \lambda' \) respectively.
Let \( q' = u^\dag qu \).  Then
\[ n^2 \lambda' = [n]_{A'}^* \lambda' = f^*u^*h^*\lambda' = f^*(\lambda \circ q').  \]
In terms of symplectic forms this says that
\[ n^2 \psi'(x, y) = \psi(f_* x, q'_* f_* y). \]

In particular, since the coordinates (with respect to \( \cB \), a symplectic basis for \( \psi \)) of \( \{ f_* e'_1, \dotsc, f_* e'_{2g} \} \) and the entries of the matrix \( q'_* \) are bounded by a polynomial in~\( n \), the same is true for
\[ \abs{\psi'(e'_i, e'_j)}. \]
Hence by \cref{symplectic-basis-bound} there is a symplectic basis \( \cB' \) for \( H_1(A', \Z) \) whose coordinates with respect to \( \{ e'_1, \dotsc, e'_{2g} \} \) are polynomially bounded.
Using again that the coordinates with respect to \( \cB \) of \( \{ f_* e'_1, \dotsc, f_* e'_{2g} \} \) are polynomially bounded, we deduce that \( H(f, \cB', \cB) \) is also polynomially bounded.
\qed

\renewcommand{\thesubsection}{\thesection.A}
\renewcommand{\thelemma}{\thesection.A}

\subsection{A bound for a period matrix}

In the published version of this paper, \cref{rat-rep-height-bound} was used to prove \cref{matrix-height-bound-fg}.
Gabriel Dill discovered during his PhD studies that the height bound of \cref{rat-rep-height-bound} is insufficient for this purpose: it is also necessary to bound the period matrix of the basis~\( \cB' \) which appears in the proposition.
This additional bound is given in \cref{correction:rat-rep-height-bound-plus} below, which was added to the arXiv version of the paper in~2016.

Note that the isogeny \( f \) in the conclusion of \cref{rat-rep-height-bound} is from \( A' \) to~\( A \), while the isogeny~\( f \) in \cref{correction:rat-rep-height-bound-plus} is from \( A \) to~\( A' \).
This change of direction is just a matter of convenience.

The proof of conclusion~(i) of \cref{correction:rat-rep-height-bound-plus} is essentially the same as the proof of \cref{rat-rep-height-bound}, apart from the minor changes due to the isogeny~\( f \) going in the opposite direction.
Conclusion~(ii) is read off directly from \cref{heights-of-endoms}.
The new part is conclusion~(iii), whose proof was supplied by Dill.

\begin{proposition} \label{correction:rat-rep-height-bound-plus}
Let \( (A, \lambda) \) be a principally polarised abelian variety over~\( \C \) and fix a symplectic basis~\( \cB \) for \( H_1(A, \Z) \).
There exist constants \( c \), \( k \) depending only on~\( (A, \lambda) \) such that:
If \( (A', \lambda') \) is any principally polarised abelian variety for which there exists an isogeny \( A \to A' \) of degree~\( n \),
then there are an isogeny \( f \colon A \to A' \) and a symplectic basis \( \cB' \) for \( H_1(A', \Z) \) such that
\begin{enumerate}[(i)]
\item \( H(f, \cB, \cB') \leq cn^k \);
\item if \( q \) is the endomorphism of~\( A \) such that \( f^* \lambda' = \lambda q \), then \( H(q, \cB, \cB) \leq cn^k \); and
\item if \( \tau \) is the period matrix of \( (A', \lambda') \) with respect to \( \cB' \), then
\[ \max(\abs{\tau_{ij}}, (\det \Im \tau)^{-1}) \leq cn^k. \]
\end{enumerate}
\end{proposition}

\begin{proof}
Let \( h \colon A \to A' \) be an isogeny of degree~\( n \), as given by the hypothesis of the proposition.
There is \( r \in \End A \) such that
\[ h^* \lambda' = \lambda \circ r. \]
Apply \cref{heights-of-endoms} to get \( u \in (\End A)^\times \) such that
\[ H(u^\dag r u) \leq c . \deg r = c n^2. \]
Letting \( f = hu \) and \( q = u^\dag r u \), we see that conclusion~(ii) of the proposition holds.

\medskip

The image of \( f_* \colon H_1(A, \Z) \to H_1(A', \Z) \) is a submodule of index~\( n \).
By the structure theory of finitely generated \( \Z \)-modules, there is a basis \( \cB'_1 \) for \( H_1(A', \Z) \) with respect to which the rational representation of~\( f \) is upper triangular and has height at most~\( n \).
But this need not be a symplectic basis.

Let \( \psi \), \( \psi' \) be the symplectic forms on \( H_1(A, \Z) \) and \( H_1(A', \Z) \) induced by \( \lambda \), \( \lambda' \) respectively.
In terms of symplectic forms, the fact that \( f^* \lambda' = \lambda q \) translates into
\[ \psi'(f_* x, f_* y) = \psi(x, q_* y) \text{ for } x, y \in H_1(A, \Z). \]
In particular, since \( H(q, \cB, \cB) \) is bounded by a polynomial in~\( n \) and noting that the values of \( \psi \) on~\( \cB \) are constant, we conclude that the values of \( \psi'(f_* x, f_* y) \) for \( x, y \in \cB \) are polynomially bounded.

The coordinates of elements of~\( \cB'_1 \) with respect to the basis \( \{ f_* x : x \in \cB \} \) are given by the inverse of the rational representation of~\( f \) (with respect to \( \cB \), \( \cB'_1 \)) and therefore have polynomially bounded height.
It follows that the values of \( \psi' \) on \( \cB'_1 \) are polynomially bounded.
Hence by \cref{symplectic-basis-bound}, there exists a symplectic basis~\( \cB' \) for \( H_1(A', \Z) \) whose coordinates with respect to \( \cB_1' \) are polynomially bounded.

Because \( H(f, \cB, \cB_1') \leq n \) and the coordinates of \( \cB' \) with respect to \( \cB_1' \) are polynomially bounded, it follows that \( H(f, \cB, \cB') \) is also polynomially bounded or in other words, conclusion~(i) of the proposition holds.

\medskip

In order to prove conclusion~(iii) of the proposition, we look at the symmetric bilinear forms \( \Phi \), \( \Phi' \) on \( H_1(A, \R) \) and \( H_1(A', \R) \) respectively, which are the real parts of the Hermitian forms induced by the polarisations \( \lambda \),~\( \lambda' \).
Just as for the symplectic forms, these satisfy the relation
\[ \Phi'(f_* x, f_* y) = \Phi(x, q_* y) \text{ for } x, y \in H_1(A, \R). \]
Because \( H(q, \cB, \cB) \) and \( H(f, \cB, \cB') \) are polynomially bounded and because the values of~\( \Phi \) on \( \cB \) are constant, we can deduce that the values of \( \Phi' \) on \( \cB' \) are bounded by a polynomial in~\( n \).

Let \( \tau = X + iY \) with \( X, Y \in \rM_{g \times g}(\R) \).
Using the definition of \( \tau \) as the period matrix of \( \cB' \), a calculation shows that the matrix of \( \Phi' \) with respect to the basis~\( \cB' \) is given by
\[ M = \begin{pmatrix}   XY^{-1}X + Y  &  XY^{-1}   \\   Y^{-1}X  &  Y^{-1}   \end{pmatrix}. \]

We have shown that the entries of~\( M \) are bounded by a polynomial in~\( n \).
Using the bottom right quadrant of~\( M \), we deduce that the entries of \( Y^{-1} \) are polynomially bounded.
It follows that \( \det Y^{-1} \) is polynomially bounded, which is part of conclusion~(iii).

Using the top left quadrant of~\( M \), we conclude that the entries of \( XY^{-1}X + Y \) are polynomially bounded, and hence \( \det(XY^{-1}X + Y) \) is polynomially bounded.
By Minkowski's determinant inequality (\cite{marcus-minc} section~II.4.1.8), using the fact that \( XY^{-1}X  \) and \( Y \) are both symmetric positive semi-definite matrices, we have
\[ \det Y \leq \det(XY^{-1}X + Y) \]
and hence \( \det Y \) is polynomially bounded.

Because entries of~\( Y^{-1} \) and \( \det Y \) are polynomially bounded, we deduce that entries of~\( Y \) are polynomially bounded.

Using the top right quadrant of~\( M \), we see that the entries of \( XY^{-1} \) are polynomially bounded.
Because we have shown that entries of \( Y \) are polynomially bounded, the same also holds for entries of~\( X \).

Combining the bounds for entries of \( X \) and~\( Y \), we conclude that entries of \( \tau \) are polynomially bounded.
Together with the fact that \( \det Y^{-1} \) is polynomially bounded, this proves conclusion~(iii) of the proposition.
\end{proof}

\renewcommand{\thesubsection}{\thesection.\arabic{subsection}}
\renewcommand{\thelemma}{\thesection.\arabic{lemma}}

\section{Isogeny theorem over finitely generated fields} \label{sec:isog-fg-fields}

The Masser--Wüstholz isogeny theorem~\cite{mw:isogeny-avs} gives a bound for the minimum degree of an isogeny between two abelian varieties over number fields, as a function of one of the varieties and the degree of their joint field of definition.
In order to prove \cref{main-thm-curves,main-thm-higher} for points \( s \in \Ag \) defined over \( \C \) and not merely over \( \bar{\Q} \), we need to extend the isogeny theorem to abelian varieties defined over finitely generated fields of characteristic~\( 0 \).
We will do this by a specialisation argument, using the fact that any abelian scheme has a closed fibre in which the specialisation map of endomorphism rings is surjective.
The proof is based on Raynaud's proof~\cite{raynaud:manin-mumford} that the Manin--Mumford conjecture over~\( \bar{\Q} \) implies the conjecture over~\( \C \).

A key feature of the theorem of Masser and Wüstholz is the explicit dependence of the bound on the abelian variety \( A \), via the Faltings height.
Our theorem does not make this explicit, and it is not apparent that there is any analogy of the Faltings height over a finitely generated field which would enable it to be made explicit.
Instead what matters to us is the dependence on the field of definition of~\( B \).

\begin{theorem} \label{isogeny-bound}
Let \( K \) be a finitely generated field of characteristic \( 0 \) and \( A \) an abelian variety defined over \( K \).
There exist constants \( c(A, K) \) and \( \kappa \) (\( \kappa \) depending only on \( \dim A \)) such that:

If \( B \) is any abelian variety defined over a finite extension \( L \) of \( K \) and isogenous over \( \bar{K} \) to \( A \),
then there exists an isogeny \( A \to B \) defined over \( \bar{K} \) of degree at most
\[ c(A, K) [L:K]^\kappa. \]
\end{theorem}

In Masser and Wüstholz's theorem, the constant \( c \) depended also on the degrees of polarisations of \( A \) and \( B \).
This dependence has been eliminated by Gaudron and Rémond~\cite{gaudron-remond:isogenies} who also showed that we can take \( \kappa = 2^{10} (\dim A)^3 + \epsilon \) for the exponent.

\begin{proof}
Let \( R \) be a finitely generated normal \( \Q \)-algebra whose field of fractions is~\( K \),
and let \( S = \Spec R \).
There is an abelian scheme \( \cA \) over some open subset \( U \subset S \) whose generic fibre is isomorphic to~\( A \).
(Note that \( R \) is a finitely generated \( \Q \)-algebra, not a finitely generated \( \Z \)-algebra, because unlike in \cite{raynaud:manin-mumford} we do not need to reduce modulo \( p \), while Noot's specialisation result requires the base to be a variety over~\( \Q \).)

By replacing \( L \) by a larger extension of bounded degree (the bound depending only on \( \dim A \)), we may assume that all homomorphisms \( A \to B \) are defined over~\( L \) (\cite{mw:periods} Lemma~2.1).
Let \( R' \) be the integral closure of \( R \) in \( L \) and \( S' = \Spec R' \).  Let \( \pi : S' \to S \) be the obvious finite morphism and let \( U' = \pi^{-1}(U) \).

Because \( A \) and \( B \) are isogenous, there is an abelian scheme \( \cB \) over~\( U' \) with generic fibre isomorphic to \( B \), and such that \( \cB \) is isogenous to \( \cA \).
We can construct this as follows: let \( N \) be the kernel of an isogeny \( A \to B \).
We can extend \( N \) to a finite flat subgroup scheme \( \cN \subset \cA \).
Then let \( \cB \) be the quotient \( \cA/\cN \).

For any closed points \( s' \in U' \) and \( s = \pi(s') \in U \), the fibres \( \cA_s \) and \( \cB_{s'} \) are abelian varieties over the number fields \( k_s \) and \( k_{s'} \), isogenous over \( k_{s'} \).
We can apply the Masser--Wüstholz theorem to deduce that there are constants \( c(\cA_s, k_s) \) and \( \kappa(\dim A) \) and an isogeny \( \cA_s \to \cB_{s'} \) of degree at most
\[ c(\cA_s, k_s) [k_{s'} : k_s]^\kappa. \]
Observe that \( [k_{s'} : k_s] \leq [L:K] \).

In order to prove the theorem, all we have to do is show that this isogeny \( \cA_s \to \cB_{s'} \) lifts to an isogeny \( A \to B \) (which will have the same degree).
Hence it will suffice to show that there is some closed point \( s \) such that the specialisation map
\[ \Hom_{\bar{K}}(A, B) \to \Hom_{\bar{k_s}}(\cA_s, \cB_{s'}) \tag{*} \label{eq:hom-specialisation-map} \]
is surjective.
Because we want a bound which depends only on \( A \) and not on \( B \), we have to show that there is a single point \( s \in U \) which will work for all \( B \).

We choose a closed point \( s \in U \) such that \( \End_{\bar{K}} A \to \End_{\bar{k_s}} \cA_s \) is surjective.  Such an \( s \) exists by \cite{noot:ordinary} Corollary~1.5 (this is proved using the Hilbert irreducibility theorem).

Let \( f_s \) be a \( \bar{k_s} \)-homomorphism \( \cA_s \to \cB_{s'} \).
To prove that \eqref{eq:hom-specialisation-map} is surjective,
we have to show that \( f_s \) lifts to \( \bar{K} \)-homomorphism \( A \to B \).

We are assuming that \( A \) and \( B \) are isogenous.
Choose any isogeny \( g_\eta : A \to B \) and let \( g_s \) be its specialisation at \( s \).
Let
\[ \alpha_s = g_s^{-1} \circ f_s \in \End_{\bar{k_s}} \cA_s \otimes_\Z \Q. \]
By our choice of \( s \), this lifts to some \( \alpha_\eta \in \End_{\bar{K}} A \otimes_\Z \Q \).
Then \( f_\eta = g_\eta \circ \alpha_\eta \) is a quasi-isogeny \( A \to B \) specialising to \( f_s \).

All we have to do is check that \( f_\eta \) is an isogeny and not just a quasi-isogeny.
Choose an integer \( m \) such that \( mf_\eta \) is an isogeny.
The kernel of \( mf_s \) contains \( \cA_s[m] \) so lifting to the generic fibre, the kernel of \( mf_\eta \) contains \( A[m] \).
Hence \( mf_\eta \) factorises as \( f'_\eta \circ [m] \) for an isogeny \( f'_\eta : A \to B \), and we must have \( f'_\eta = f_\eta \).
\end{proof}

\bibliographystyle{amsalpha}
\bibliography{andre-pink2}

\end{document}